\documentclass[amstex,12pt, amssymb]{article}

\usepackage{mathtext}
\usepackage[cp1251]{inputenc}
\usepackage[T2A]{fontenc}
\usepackage[dvips]{graphicx}
\usepackage{amsmath}
\usepackage{amssymb}
\usepackage{amsxtra}
\usepackage{latexsym}
\usepackage{ifthen}

\newcounter{lemma}[section]

\newcounter{corollary}[section]

\newcounter{remark}[section]

\newcounter{theorem}[section]

\newcounter{proposition}[section]

\newcounter{example}

\numberwithin{equation}{section}

\begin{document}

\markboth{E.~SEVOST'YANOV}{\centerline{ON THE INVERSE POLETSKY
INEQUALITY ...}}

\def\cc{\setcounter{equation}{0}
\setcounter{figure}{0}\setcounter{table}{0}}

\overfullrule=0pt


\author{OLEKSANDR DOVHOPIATYI, EVGENY SEVOST'YANOV}

\title{
{\bf ON THE INVERSE $K_I$-INEQUALITY FOR ONE CLASS OF MAPPINGS}}

\date{\today}
\maketitle

\begin{abstract}
We study mappings differentiable almost everywhere, possessing the
$N$-Luzin property, the $ N^{\,-1}$-property on the spheres with
respect to the $(n-1)$-dimensional Hausdorff measure and such that
the image of the set where its Jacobian equals to zero has a zero
Lebesgue measure. It is proved that such mappings satisfy the lower
bound for the Poletsky-type distortion in their domain of
definition.
\end{abstract}


\section{Introduction}

One of the methods of studying the Sobolev and Orlicz-Sobolev
classes is to use the distortion estimates of the modulus of
families of paths and surfaces (see, for example, \cite {KRSS} and
\cite{Sev$_1$}). In particular, the lower estimates of the
distortion for the modulus of families of images of concentric
spheres under the mapping have an important role in the study of
their local and boundary behavior, see ibid. Note that, in the
mentioned papers, we are talking only about the mapped surfaces,
while estimates of the modulus of the families of these surfaces
themselves were not involved, as their role has not been studied in
detail. The main purpose of this manuscript is to obtain the
estimates of modulus of families of sets, the image of which under
the map are spheres centered at a fixed point. As will be shown
below, these estimates associated with the so-called inverse
Poletsky inequality, which makes it possible to describe many
properties of the corresponding mappings with taking into account
our previous results (see, e.g., \cite{SSD}).

\medskip
Here are the necessary definitions and wording of the main result.
Let $X$ and $ Y $ be two spaces with measures $\mu$ and
$\mu^{\,\prime},$ respectively. We say that a mapping
$f:X\rightarrow Y$ has {\it $N$-property of Luzin}, if from the
condition $\mu(E)=0$ it follows that $\mu^{\,\prime}(f(E))=0.$
Similarly, we say that a mapping $f:X\rightarrow Y$ has {\it $N^{\,
\prime}$-Luzin property,} if from the condition
$\mu^{\,\prime}(E)=0$ it follows that $\mu(f^{\,-1}(E))=0.$ At the
points $x \in D$ of differentiability of the  mapping $f,$ we put

$$l(f^{\,\prime}(x))\,=\,\min\limits_{h\in {\Bbb R}^n
\backslash \{0\}} \frac {|f^{\,\prime}(x)h|}{|h|}\,,$$
\begin{equation}\label{eq5_a}
\Vert f^{\,\prime}(x)\Vert\,=\,\max\limits_{h\in {\Bbb R}^n
\backslash \{0\}} \frac {|f^{\,\prime}(x)h|}{|h|}\,,
\end{equation}
$$J(x,f)=\det
f^{\,\prime}(x)\,.$$
Fix $p>1.$ We define the {\it inner} and the {\it outher}
dilatations of the mapping $f$ at a point $x$ of the order $p$ by
the relations
$$K_{I, p}(x,f)\quad =\quad\left\{
\begin{array}{rr}
\frac{|J(x,f)|}{{l(f^{\,\prime}(x))}^p}, & J(x,f)\ne 0,\\
1,  &  f^{\,\prime}(x)=0, \\
\infty, & \text{otherwise}
\end{array}
\right.\,, $$$$K_{O, p}(x,f)\quad =\quad \left\{
\begin{array}{rr}
\frac{\Vert f^{\,\prime}(x)\Vert^p}{|J(x,f)|}, & J(x,f)\ne 0,\\
1,  &  f^{\,\prime}(x)=0, \\
\infty, & \text{otherwise}
\end{array}
\right.\,\,,$$
respectively. Given a mapping $f:D\,\rightarrow\,{\Bbb R}^n,$ a set
$E\subset D$ and $y\,\in\,{\Bbb R}^n,$ we define the {\it
multiplicity function $N(y,f,E)$} as a number of preimages of the
point $y$ in a set $E,$ i.e.
$$
N(y,f,E)\,=\,{\rm card}\,\left\{x\in E: f(x)=y\right\}\,,
$$
$$
N(f,E)\,=\,\sup\limits_{y\in{\Bbb R}^n}\,N(y,f,E).
$$
Let $A$ be a set where $f$ does not have a total differential, and
let $y\not\in f(A).$ If $N(f, D)<\infty,$ then we set
\begin{equation}\label{eq1}
Q(y):=K_{I, \alpha}(y, f^{\,-1})=\sum\limits_{x\in
f^{\,-1}(y)}{K_{O, \alpha}(x, f)}\,.
\end{equation}
Observe that, $N(f, D)<\infty$ for open, discrete and closed
mappings of $D,$ see \cite[Lemma~3.3]{MS}.

\medskip
Let $y_0\in {\Bbb R}^n,$ $0<r_1<r_2<\infty$ and
\begin{equation}\label{eq1**}
A=A(y_0, r_1,r_2)=\left\{ y\,\in\,{\Bbb R}^n:
r_1<|y-y_0|<r_2\right\}\,.\end{equation}
Given sets $E,$ $F\subset\overline{{\Bbb R}^n}$ and a domain
$D\subset {\Bbb R}^n$ we denote by $\Gamma(E,F,D)$ a family of all
paths $\gamma:[a,b]\rightarrow \overline{{\Bbb R}^n}$ such that
$\gamma(a)\in E,\gamma(b)\in\,F$ and $\gamma(t)\in D$ for $t \in [a,
b].$ Given a mapping $f:D\rightarrow {\Bbb R}^n,$ a point $y_0\in
\overline{f(D)}\setminus\{\infty\},$ and
$0<r_1<r_2<r_0=\sup\limits_{y\in f(D)}|y-y_0|,$ we denote by
$\Gamma_f(y_0, C_1, C_2)$ a family of all paths $\gamma$ in $D$ such
that $f(\gamma)\in \Gamma(C_1, C_2, A(y_0,r_1,r_2)).$ Let $Q_*:{\Bbb
R}^n\rightarrow [0, \infty]$ be a Lebesgue measurable function, and
$M_{\alpha}(\Gamma)$ denotes the $\alpha$-modulus od a family
$\Gamma$ (see, e.g.,~\cite[section~6]{Va}). We say that {\it $f$
satisfies the inverse Poletsky inequality at a point $y_0\in
\overline{f(D)}\setminus\{\infty\}$ with respect to
$\alpha$-modulus} if there is $r_0>0$ such that, the relation
\begin{equation}\label{eq2*A}
M_{\alpha}(\Gamma_f(y_0, C_1, C_2))\leqslant
\int\limits_{A(y_0,r_1,r_2)\cap f(D)} Q_*(y)\cdot \eta^{\alpha}
(|y-y_0|)\, dm(y)
\end{equation}
holds for any $0<r_1<r_2<r_0,$ any continua $C_1\subset
\overline{B(y_0, r_1)}\cap f(D)$ and $C_2\subset f(D)\setminus
B(y_0, r_2),$ and any Lebesgue measurable function $\eta:
(r_1,r_2)\rightarrow [0,\infty ]$ such that
\begin{equation}\label{eqA2}
\int\limits_{r_1}^{r_2}\eta(r)\, dr\geqslant 1\,.
\end{equation}
The following statement holds.

\medskip
\begin{theorem}\label{th1} {\sl Let $n-1<\alpha\leqslant n,$ let $y_0\in
\overline{f(D)}\setminus\{\infty\},$ $r_0=\sup\limits_{y\in
f(D)}|y-y_0|>0,$ and let $f:D\rightarrow{\Bbb R}^n$ be an open,
discrete and closed mapping that is differentiable almost everywhere
and has $N$-Luzin property with respect to the Lebesgue measure in
${\Bbb R}^n.$ Suppose that $\overline{D}$ is a compact set in ${\Bbb
R}^n,$ and, in addition,
\begin{equation}\label{eq2}
m (f\left(\left\{x\in D: J(x, f)=0\right\}\right))=0\,.
\end{equation}
Suppose that $f$ has $N^{\,- 1}$-property on $S(y_0, r)\cap f(D)$
for almost all $r\in(\varepsilon, r_0)$ relative to the Hausdorff
measure ${\mathcal H}^{n-1}$ on $ S(y_0, r).$ If the function $Q,$
which is defined in~(\ref{eq1}), belongs to the class $L^{1}(f(D)),$
then the mapping $f$ satisfies the inverse Poletsky inequality with
respect to $\alpha$-modulus with $Q_*(y):=N^{\alpha}(f, D)\cdot
Q(y).$
 }
\end{theorem}

\medskip
\begin{corollary}\label{cor2}
{\sl The assertion of Theorem~\ref{th1} holds if instead of the
condition~(\ref{eq2}) a stronger condition is required: $J(x, f) \ne
0$ almost everywhere.}
\end{corollary}

\section{Distortion of families of sets under mappings}

Let us give some important information concerning the relationship
between the moduli of the families of paths joining the sets and the
moduli of the families of the sets separating these sets. Mostly
this information can be found in Ziemer's publication,
see~\cite{Zi$_1$}. Let $G$ be a bounded domain in ${\Bbb R}^ n,$ and
$C_0, C_1$ are disjoint compact sets in $\overline{G}.$ Put $R=G
\setminus (C_{0} \cup C_{1})$ and $R^{\,*}=R \cup C_{0}\cup C_{1}.$
For a number $p>1 $ we define a {\it $p$ -capacity of the pair $C_0,
C_1 $ relative to the closure $G$} by the equality
$$C_p[G, C_0, C_1] = \inf \int\limits_{R} |\nabla u|^p\, dm(x),$$
where the exact lower bound is taken for all functions $u,$
continuous in $R^{\,*},$ $u\in ACL(R),$ such that $u=1$ on $C_1$ and
$u=0$ on $C_0.$ These functions are called {\it admissible} for
$C_p[G, C_0, C_1].$ We say that a set $\sigma \subset {\Bbb R}^n$
{\it separates} $C_0$ and $C_1$ in $R^{\,*},$ if $\sigma \cap R$ is
closed in $R$ and there are disjoint sets $A$ and $B,$ open relative
$R^{\,*}\setminus \sigma,$ such that $R^{\,*}\setminus \sigma=A\cup
B,$ $C_0\subset A$ and $C_1\subset B.$ Let $\Sigma$ denotes the
class of all sets that separate $C_0$ and $C_1$ in $R^{\,*}.$ For
the number $p^{\prime}=p/(p- 1)$ we define the quantity
\begin{equation}\label{eq13.4.12}
\widetilde{M_{p^{\prime}}}(\Sigma)=\inf\limits_{\rho\in
\widetilde{\rm adm} \Sigma} \int\limits_{{\Bbb
R}^n}\rho^{\,p^{\prime}}dm(x)
\end{equation}
where the notation $\rho\in \widetilde{\rm adm}\,\Sigma$ denotes
that $\rho$ is nonnegative Borel function in ${\Bbb R}^n$ such that
\begin{equation} \label{eq13.4.13}
\int\limits_{\sigma \cap R}\rho\, d{\mathcal H}^{n-1} \geqslant
1\quad\forall\, \sigma \in \Sigma\,. \end{equation}
Note that according to the result of Ziemer
\begin{equation}\label{eq3}
\widetilde{M_{p^{\,\prime}}}(\Sigma)=C_p[G , C_0 ,
C_1]^{\,-1/(p-1)}\,,
\end{equation}
see~\cite[Theorem~3.13]{Zi$_1$} for $p=n$ and \cite[p.~50]{Zi$_2$}
for $1<p<\infty,$ in addition, by the Hesse result
\begin{equation}\label{eq4}
M_p(\Gamma(E, F, D))= C_p[D, E, F]\,,
\end{equation}
where $(E \cup F)\cap
\partial D = \varnothing$ (see~\cite[Theorem~5.5]{Hes}). Shlyk has proved that the
requirement $(E \cup F)\cap
\partial D = \varnothing$ can be omitted, in other words, the equality~(\ref{eq4})
holds for any disjoint non-empty sets $E, F\subset \overline{D}$
(see~\cite[Theorem~1]{Shl}).

\medskip
Let $S$ be a surface, in other words, $S:D_s\rightarrow {\Bbb R}^n$
be a continuous mapping of an open set $D_s\subset {\Bbb R}^{n-1}.$
We put
$ N(y, S)={\rm card}\, S^{-1}(y)={\rm card} \{x\in D_s: S(x)=y\} $
and recall this function a {\it multiplicity function} of the
surface $S$ with respect to a point $y\in{\Bbb R}^n.$ Given a Borel
set $B\subset {\Bbb R}^n,$ its $(n-1)$-measured Hausdorff area
associated with the surface $S$ is determined by the formula
$ {\mathcal A}_S(B)={\mathcal A}_S^{n-1}(B)= \int\limits_B N(y, S)\,
d{\mathcal H}^{n-1} y, $
see~\cite[item~3.2.1]{Fe}. For a Borel function $\rho:\,{\Bbb
R}^n\rightarrow [0, \ infty]$ its integral over the surface $S$ is
determined by the formula
$ \int\limits_S \rho\, d{\mathcal A}=\int\limits_{{\Bbb R}^n}
\rho(y) N(y, S)\,d{\mathcal H}^{n-1} y. $
In what follows, $J_kf(x)$ denotes the {\it $k$-dimensional
Jacobian} of the mapping $f$ at a point $x$ (see \cite[$\S\,3.2,$
Ch.~3]{Fe}).

\medskip
Let $n\geqslant 2,$ and let $\Gamma$  be a family of surfaces $S.$ A
Borel function $\rho\colon{\Bbb R}^n\rightarrow\overline{{\Bbb
R}^+}$ is called {\it an admissible} for $\Gamma,$ abbr.
$\rho\in{\rm adm}\,\Gamma,$ if
\begin{equation}\label{eq8.2.6}\int\limits_S\rho^{n-1}\,
d{\mathcal{A}}\geqslant 1\end{equation} for any $S\in\Gamma.$ Given
$p\in(1,\infty),$ a {\it $p$-modulus} of $\Gamma$ is called the
quantity
$$M_p(\Gamma)=\inf_{\rho\in{\rm adm}\,\Gamma} \int\limits_{{\Bbb
R}^n}\rho^p(x)\,dm(x)\,.$$ We also set
$M(\Gamma):=M_n(\Gamma).$
Let us say that some property $P$ holds for {\it $p$-almost all
surfaces} of the domain $D,$ if this property holds for all surfaces
in $D,$ except, maybe be, some of their subfamily, $p$ -modulus of
which is zero. If we are talking about the conformal modulus
$M(\Gamma):=M_n(\Gamma),$ the prefix ''$n$'' in the expression
''$n$-almost all'', as a rule, is omitted. We say that a Lebesgue
measurable function $\rho\colon{\Bbb R}^n\rightarrow\overline{{\Bbb
R}^+}$ is {\it $p$-extensively admissible} for the family $\Gamma$
of surfaces $S$ in ${\Bbb R}^n,$ abbr. $\rho\in{\rm ext}_p\,{\rm
adm}\,\Gamma,$ if the relation~(\ref{eq8.2.6}) is satisfied for
$p$-almost all surfaces $S$ of the family $\Gamma.$ The proof of the
following lemma is based on the approach, used in establishing the
relationship of Orlicz-Sobolev classes with lower estimates of the
distortion of the modulus of surface families (see, eg,
\cite[Theorem~5]{KRSS} and \cite[Theorem~4]{Sev$_1$}). In such a
general formulation, this lemma is proved for the first time in this
paper.

\medskip
\begin{lemma}\label{lem1} {\sl Let $p>n-1,$ $f:D\rightarrow{\Bbb R}^n$
be a mapping that is differentiable almost everywhere and has
$N$-Luzin property with respect to the Lebesgue measure in ${\Bbb
R}^n,$ let $N(f, D)<\infty$ and let $y_0\in
\overline{f(D)}\setminus\{\infty\},$ $r_0=\sup\limits_{y\in
f(D)}|y-y_0|,$ $0<\varepsilon_0<r_0,$ $0<\varepsilon<\varepsilon_0.$
Suppose that the condition~(\ref{eq2}) is also satisfied. Fix
$\varepsilon>0,$ and denote by $\Sigma_{\varepsilon}$ the family of
all sets of the form
\begin{equation}\label{eq12}
\{f^{\,-1}(S(y_0, r)\cap f(D))\},\quad r\in (\varepsilon, r_0)\,.
\end{equation}
Suppose, in addition, that $f$ has $N^{\,- 1}$-property on $ S(y_0,
r)\cap f(D)$ for almost all $r\in(\varepsilon, r_0 )$ relative to
the Hausdorff measure ${\mathcal H}^{n-1}$ on $S(y_0, r).$ Then
\begin{equation}\label{eq13}
\widetilde{M}_{\frac{p}{n-1}}(\Sigma_{\varepsilon})\geqslant\frac{1}{N^{\frac{p}{n-1}}(f,
D)}\inf\limits_{\rho\in{\rm
ext\,adm}_p\,f(\Sigma_{\varepsilon})}\int\limits_{f(D)\cap A(y_0,
\varepsilon,
r_0)}\frac{\rho^p(y)}{Q^{\frac{p-n+1}{n-1}}(y)}\,dm(y)\,,
\end{equation}
where \begin{equation}\label{eq1C} Q(y):=K_{I, \alpha}(y,
f^{\,-1})=\sum\limits_{x\in f^{\,-1}(y)}{K_{O, \alpha}(x, f)}\,,
\end{equation}
and $\alpha=\frac{p}{p-n+1}.$}
\end{lemma}

\begin{proof} Without loss of generality, we may assume that $r_0>0.$
We will generally follow the methodology set forth in proving
\cite[Theorem~5]{KRSS} (see also \cite[Theorem~8.6]{MRSY}).

\medskip
Denote by $B$ a Borel set of all points $x\in D,$ where the mapping
$f$ has a total differential $f^{\,\prime}(x)$ and $J(x, f)\ne 0.$
By Kirsbraun's theorem and by the unity of the approximate
differential (see, for example, \cite[2.10.43 and
Theorem~3.1.2]{Fe}) it follows that the set $B$ is a countable union
of Borel sets $B_k,$ $k=1,2, \ldots ,$ such that the mappings
$f_k=f|_{B_k}$ are Bilipschitz homeomorphisms (see \cite[Lemma~3.2.2
and Theorems 3.1.4 and 3.1.8]{Fe}). Without loss of generality, we
may assume that the sets $B_k$ are disjoint. We also denote by $B_*$
the set of all points $x\in D,$ where $f$ has a total differential,
but $J(x, f)=0.$

\medskip Since the set $B_0:=D\setminus (B\cup B_*)$ has a Lebesgue measure zero,
and the mapping $f$ has $N$-Luzin property, then $m(f(B_0))= 0.$
By~\cite[Theorem~9.3]{MRSY} ${\mathcal A}_{S_r}(f(B_0))=0$ for
$p$-almost all spheres $S_r:=S(y_0, r)\cap f (D)$ centered at a
point $y_0,$ where ''almost all'' is understood in the sense of
$p$-modulus of families of surfaces. Note that, the function
$\psi(r):={\mathcal H}^{n-1}(f (B_0)\cap S_r)$ is Lebesgue due to
the Fubini theorem  (\cite[Section~8.1, Ch.~III]{Sa}). Thus, the set
$E\subset {\Bbb R} $ of all $r\in {\Bbb R}$ such that ${\mathcal
H}^{n-1}(f(B_0)\cap S_r)=0,$ is Lebesgue measurable. Then by
\cite[Lemma~4.1]{IS} ${\mathcal A}_{S_r}(f(B_0))=0$ for almost all
spheres $S_r:=S(y_0, r)$ centered at the point $y_0,$ where ''almost
all'' is understood in the sense of a one-dimensional Lebesgue
measure with respect to the parameter $r\in (\varepsilon, r_0).$
Now, by the assumption of Lemma,
\begin{equation}\label{eq3B}
{\mathcal H}^{n-1}(f^{\,-1}(S_r)\cap B_0)=0
\end{equation}
for almost all $r\in (\varepsilon, \varepsilon_0).$ Arguing
similarly, we obtain that
\begin{equation}\label{eq4B}
{\mathcal H}^{n-1}(f^{\,-1}(S_r)\cap B_*)=0
\end{equation}
for almost all $r\in (\varepsilon, \varepsilon_0).$

\medskip
Let $\rho^{n-1}\in \widetilde{{\rm adm}}\,\Sigma_{\varepsilon}$ and
let
\begin{equation}\label{eq8.3.12}\tilde{\rho}(y)\ =\begin{cases}
\sup\limits_{x\in f^{\,-1}(y)\cap D\setminus B_0}\rho_*(x)\,,& y\in
f(D)\setminus f(B\cap B_*)\,\\
0\,,&y\in f(B\cap B_*)
\end{cases}\,,
\end{equation}
where
\begin{equation}\label{eq8.3.13}\rho_*(x)=\left \{\begin{array}{rr}
\rho(x)\cdot \left(\frac{\Vert f^{\,\prime}(x)\Vert}{J(x, f)}
\right)^{1/(n-1)}, &  x\in D\setminus B_0, \\
0, & \text{otherwise}\end{array}\right. \end{equation}
Observe that $\tilde{\rho}=\sup\rho_k,$ where
\begin{equation}\label{eq8.3.14} \rho_k(y)\ =\ \left
\{\begin{array}{rr} \rho_*(f^{\,-1}_k(y)), &{\rm при }\ y\in f(B_k),\\
0, & \text{otherwise}\end{array}\right.\end{equation} and, moreover,
each mapping $f_k=f|_{B_k},$ $k=1,2,\ldots ,$ is injective. Thus, a
function $\tilde{\rho}$ is Borel (see, e.g.,
\cite[Theorem~I~(8.5)]{Sa}).

Let $f^{\,-1}(S_r):=S^{\,*}_r.$ Then
$$\int\limits_{S_r\cap f(D)}{\widetilde{\rho}}^{\,n-1}(y)\,d{\mathcal A_*}=
\int\limits_{{\Bbb R}^n}{\widetilde{\rho}}^{\,n-1}(y)\chi_{S_r\cap
f(D)}(y)\,\,d{\mathcal H}^{n-1}y\geqslant
$$
$$\geqslant\int\limits_{{\Bbb R}^n}\frac{1}{N(f,
D)}\cdot\sum\limits_{k=1}^{\infty}{\widetilde{\rho}}^{\,n-1}(y)\chi_{S_r\cap
f(D)}(y) N(y, f, B_k\cap S^{\,*}_r)\,d{\mathcal H}^{n-1}y=$$
\begin{equation}\label{eq1A}
=\frac{1}{N(f, D)}\sum\limits_{k=1}^{\infty}\int\limits_{{\Bbb
R}^n}{\rho_*}^{n-1}(f_k^{\,-1}(y))N(y, f, B_k\cap
S^{\,*}_r)\,d{\mathcal H}^{n-1}y=
\end{equation}
$$= \frac{1}{N(f,
D)}\sum\limits_{k=1}^{\infty} \int\limits_{f(B_k\cap
S^{\,*}_r)}{\rho_*}^{n-1}(f_k^{\,-1}(y))\,d{\mathcal H}^{n-1}y\,.$$

Let $\lambda_1(x),\lambda_2(x),\ldots, \lambda_n(x)$ are the main
stretchings of the mapping $f,$ see e.g. \cite[Lem\-mas~4.1.I,
4.2.I]{Re}. Then $J(x, f)=\lambda_1(x)\cdots\lambda_n(x)$ and
\begin{equation}\label{eq1B}
\left(\frac{\Vert f^{\,\prime}(x)\Vert}{J(x, f)}
\right)^{1/(n-1)}=\left(\frac{1}{\lambda_1(x)\ldots
\lambda_{n-1}(x)}\right)^{\frac{1}{n-1}}\geqslant
\left(\frac{1}{J_{n-1}f(x)}\right)^{\frac{1}{n-1}}\,.
\end{equation}
Due to~(\ref{eq3B}), (\ref{eq4B}) and~(\ref{eq1B}), by
\cite[Corollary~3.2.20]{Fe} for $m=n-1,$ we obtain that
$$
\sum\limits_{k=1}^{\infty}\int\limits_{f(B_k\cap
S^{\,*}_r)}{\rho_*}^{n-1}(f_k^{\,-1}(y))\,d{\mathcal
H}^{n-1}y=\sum\limits_{k=1}^{\infty}\int\limits_{B_k\cap S^{\,*}_r}
{\rho_*}^{n-1}(x)\,\,J_{n-1}f(x)\,d{\mathcal H}^{n-1}x\,=$$
$$=\sum\limits_{k=1}^{\infty}\int\limits_{B_k\cap
S^{\,*}_r} \frac{\rho^{n-1}(x)\Vert f^{\,\prime}(x)\Vert}{J(x,
f)}\,J_{n-1}f(x)\,d{\mathcal H}^{n-1}x\,\geqslant $$
\begin{equation}\label{eq3A}
\geqslant \sum\limits_{k=1}^{\infty}\int\limits_{B_k\cap
S^{\,*}_r}\rho^{n-1}(x)\, \,d{\mathcal
H}^{n-1}x=\int\limits_{f^{\,-1}(S_r)}\rho^{n-1}(x) \,d{\mathcal
H}^{n-1}x\geqslant 1
\end{equation}
for almost any $S_r=f\circ S^{\,*}_r\in f(\Sigma_{\varepsilon}).$ It
follows from~(\ref{eq3A}) that $N^{\frac{1}{n-1}}(f,
D)\tilde{\rho}\in {\rm ext\,adm}_p\,f(\Sigma_{\varepsilon})$
(see~\cite[Lemma~4.1]{IS}).

\medskip
Since  $\tilde{\rho}^p(y)=\sup\limits_{k\in {\Bbb
N}}\rho^p_k(y)\leqslant\sum\limits_{k=1}^{\infty}\rho^p_k(y)$ and
$m(f(B_*))=m(f(B_0))=0,$ then
$$\int\limits_{f(D)}\frac{\tilde{\rho}^p(y)}{Q(y)}\,dm(y)\leqslant
\sum\limits_{k=1}^{\infty}\int\limits_{f(B_k)}
\frac{\rho^p_k(y)}{Q(y)}\,dm(y)\leqslant
\sum\limits_{k=1}^{\infty}\int\limits_{f(B_k)}
\frac{\rho^p_k(y)}{K^{\frac{p-n+1}{n-1}}_{O, \alpha}(f_k^{\,-1}(y),
f)}\,dm(y)\,.$$
Using the change of variables formula on each $B_k,$ $k=1,2,
\ldots,$ see, for example, \cite[Theorem~3.2.5]{Fe}, we obtain that
$$\int\limits_{f(B_k)}
\frac{\rho^p_k(y)}{K^{\frac{p-n+1}{n-1}}_{O, \alpha}(f_k^{\,-1}(y),
f)}\,dm(y)=$$$$=
\int\limits_{f(B_k)}\frac{\rho^p(f_k^{\,-1}(y))J^{\frac{p-n+1}{n-1}}(f_k^{\,-1}(y),
f)} {{\Vert f^{\,\prime}(f_k^{\,-1}(y))\Vert}^{\frac{p}{p-n+1}\cdot
\frac{p-n+1}{n-1}}}\cdot \frac{{\Vert
f^{\,\prime}(f_k^{\,-1}(y))\Vert}^{\frac{p}{n-1}}}{|J(f_k^{\,-1}(y),
f)|^{\frac{p}{n-1}}}\,dm(y)=$$
$$=\int\limits_{f(B_k)}\rho^p(f_k^{\,-1}(y))J(y,
f_k^{\,-1})\,dm(y)=\int\limits_{B_k}\rho^p(x)\,dm(x)\,.$$
The latter implies that
\begin{equation}\label{eq4A}
\int\limits_{f(D)}\frac{{\widetilde{\rho}}^p(y)}{Q^{\frac{p-n+1}{n-1}}(y)}\,dm(y)
\leqslant
\sum\limits_{k=1}^{\infty}\int\limits_{B_k}\rho^p(x)\,dm(x)\,.
\end{equation}
Summing~(\ref{eq4A}) by $k=1,2, \ldots $ and using the countable
additivity of the Lebesgue integral (see, for example,
\cite[Theorem~I.12.3]{Sa}), we obtain that
\begin{equation}\label{eq5A}
\int\limits_{f(D)}\frac{1}{Q^{\frac{p-n+1}{n-1}}(y)}{\widetilde{\rho}}^p(y)\cdot\,dm(y)
\leqslant\int\limits_D\rho^p(x)\,dm(x)\,.
\end{equation}
Going in the ratio~(\ref{eq5A}) to $\inf$ over all functions
$\rho^{n-1} \in \widetilde{{\rm adm}}\,\Sigma_{\varepsilon},$ we
obtain that
$$\int\limits_{f(D)}\frac{1}{Q^{\frac{p-n+1}{n-1}}(y){\widetilde{\rho}}^p(y)}\cdot\,dm(y)
\leqslant\widetilde{M_{\frac{p}{n-1}}}(\Sigma_{\varepsilon})\,,$$
whence we obtain that
$$\int\limits_{f(D)}\frac{N^{\frac{p}{n-1}}(f, D)}{Q^{\frac{p-n+1}{n-1}}(y)}
{\widetilde{\rho}}^p(y)\cdot\,dm(y) \leqslant N^{\frac{p}{n-1}}(f,
D)\cdot\widetilde{M_{\frac{p}{n-1}}}(\Sigma_{\varepsilon})\,.$$
Put $\widetilde{\rho}_1(y):=N^{\frac{1}{n-1}}(f, D)\cdot
\widetilde{\rho}(y).$ Due to the latter relation, we obtain that
\begin{equation}\label{eq10C}
\int\limits_{f(D)}
\frac{{\widetilde{\rho}_1}^p(y)}{Q^{\frac{p-n+1}{n-1}}(y)}\,dm(y)
\leqslant N^{\frac{p}{n-1}}(f,
D)\cdot\widetilde{M_{\frac{p}{n-1}}}(\Sigma_{\varepsilon})\,.
\end{equation}
Since by the above $\widetilde{\rho}_1(y)=N^{\frac{1}{n-1}}(f,
D)\tilde{\rho}\in {\rm ext\,adm}_p\,f(\Sigma_{\varepsilon}),$ it
follows from~(\ref{eq10C}) that the relation~(\ref{eq13}) holds.
Lemma is proved. $\Box$
\end{proof}

\medskip
We have the following simple consequence.

\medskip
\begin{corollary}\label{cor1}
{\sl Let $f:D\rightarrow {\Bbb R}^n$ be a map which is
differentiable almost everywhere, and has $N$ and $N^ {\,-1}$ Luzin
properties with respect to the Lebesgue measure. Let $y_0\in
\overline{f(D)}\setminus\{\infty\},$ $r_0=\sup\limits_{y\in
f(D)}|y-y_0|.$ We fix $\varepsilon>0,$ and denote by
$\Sigma_{\varepsilon}$ the family of all sets of the
form~(\ref{eq12}). In addition, suppose that $f$ has
$N^{\,-1}$-Luzin property on $S(y_0, r)\cap f(D)$ for almost all
$r\in (\varepsilon, \varepsilon_0)$ with respect to ${\mathcal
H}^{n-1}$ on $S(y_0, r).$ Then the relation~(\ref{eq13}) is
fulfilled,  where $Q$ is defined by the relation~(\ref{eq1}).}
\end{corollary}

\medskip
\begin{proof}
Since $f$ has $N^{\,- 1}$-Luzin property, by Ponomarev's theorem we
have that $J (x, f)\ne 0$ almost everywhere (see, for example,
\cite[Theorem~1]{Pon}), we may assume that $J(x, f)\ne 0 $ on any
$B_k,$ $k=1,2, \ldots. $ Then, since the mapping $f$ has
$N$-property, the condition~(\ref{eq2}) is also fulfilled. The
desired statement, in this case, follows from
Lemma~\ref{lem1}.~$\Box$
\end{proof}

\section{Proof of the main result} Let $Q_*:D\rightarrow
[0,\infty]$ be a Lebesgue measurable function. Denote by
$q_{x_0}(r)$ the integral average of $Q_*(x)$ under the sphere
$|x-x_0|=r,$
\begin{equation}\label{eq3.1A}
q_{x_0}(r):=\frac{1}{\omega_{n-1}r^{n-1}}\int\limits_{|x-x_0|=r}Q_*(x)\,d\mathcal{H}^{n-1}\,,
\end{equation}
where $\omega_{n-1}$ denotes the area of the unit sphere in ${\Bbb
R}^n.$
Below we also assume that the following standard relations hold:
$a/\infty=0$ for $a\ne\infty,$ $a/0=\infty$ for $a> 0$ and $0\cdot
\infty=0$ (see, e.g., \cite[$\S\,3,$ section~I]{Sa}). The following
conclusion was obtained by V.~Ryazanov together with the author in
the case $p=n,$ see, e.g.,~\cite[Lemma~7.4]{MRSY} or
\cite[Lemma~2.2]{RS}. In the case of an arbitrary $p> 1,$ see, for
example, \cite[Lemma~2]{SalSev}.

\medskip
\begin{proposition}\label{pr1}
{\sl\, Let $p>1,$ $n\geqslant 2,$ $x_0 \in {\Bbb R}^n,$ $r_1, r_2\in
{\Bbb R},$ $r_1, r_2>0,$ and let $Q_*(x)$ be a Lebesgue measurable
function, $Q_*:{\Bbb R}^n\rightarrow [0, \infty],$ $Q_*\in L_{\rm
loc}^1({\Bbb R}^n).$ We put
$$
I=I(x_0,r_1,r_2)=\int\limits_{r_1}^{r_2}
\frac{dr}{r^{\frac{n-1}{p-1}}q_{x_0}^{\frac{1}{p-1}}(r)}\,,
$$
and let $q_{x_0}(r)$ be defined by~(\ref{eq3.1A}). Then
\begin{equation}\label{eq10A}
\frac{\omega_{n-1}}{I^{p-1}}\leqslant \int\limits_A Q_*(x)\cdot
\eta^p(|x-x_0|)\,dm(x)
\end{equation}
for any Lebesgue measurable function $\eta :(r_1,r_2)\rightarrow
[0,\infty]$ such that
\begin{equation}\label{eq10B} \int\limits_{r_1}^{r_2}\eta(r)\,dr=1\,,
\end{equation}
where $A=A(x_0, r_1,r_2)$ is defined in~(\ref{eq1**}).}
\end{proposition}

\medskip
\begin{remark}\label{rem1}
Note that, if~(\ref{eq10A}) holds for any function $\eta$ with a
condition (\ref{eq10B}), then the same relationship holds for any
function $\eta$ with the condition~(\ref{eqA2}). Indeed, let $\eta$
be a nonnegative Lebesgue function that satisfies the condition
(\ref{eqA2}). If $J:=\int\limits_{r_1}^{r_2}\eta(t)\,dt<\infty,$
then we put $\eta_0:=\eta/J.$ Obviously, the function $\eta_0$
satisfies condition~(\ref{eq10B}). Then the relation~(\ref{eq10A})
gives that
$$\frac{\omega_{n-1}}{I^{p-1}}\leqslant \frac{1}{J^p}\int\limits_A Q_*(x)\cdot
\eta^p(|x-x_0|)\,dm(x)\leqslant \int\limits_A Q_*(x)\cdot
\eta^p(|x-x_0|)\,dm(x)$$
because $J\geqslant 1.$ Let now $J=\infty.$ Then, by
\cite[Theorem~I.7.4]{Sa}, a function $\eta$ is a limit of a
nondecreasing nonnegative sequence of simple functions $\eta_m,$
$m=1,2,\ldots .$ Set
$J_m:=\int\limits_{r_1}^{r_2}\eta_m(t)\,dt<\infty$ and
$w_m(t):=\eta_m(t)/J_m.$ Then, it follows from~(\ref{eq10B}) that
\begin{equation}\label{eq11A}
\frac{\omega_{n-1}}{I^{p-1}}\leqslant \frac{1}{J_m^p}\int\limits_A
Q_*(x)\cdot \eta_m^p(|x-x_0|)\,dm(x)\leqslant \int\limits_A
Q_*(x)\cdot \eta_m^p(|x-x_0|)\,dm(x)\,,
\end{equation}
because $J_m\rightarrow J=\infty$ as $m\rightarrow\infty$
(see~\cite[Lemma~I.11.6]{Sa}). Thus, $J_m\geqslant 1$ for
sufficiently large $m\in {\Bbb N}.$ Observe that, a functional
sequence $\psi_m(x)=Q_*(x)\cdot \eta_m^p(|x-x_0|),$ $m=1,2\ldots ,$
is nonnegative, monotone increasing and converges to a function
$\psi(x):=Q_*(x)\cdot \eta^p(|x-x_0|)$ almost everywhere. By the
Lebesgue theorem on the monotone convergence
(see~\cite[Theorem~I.12.6]{Sa}), it is possible to go to the limit
on the right side of the inequality~(\ref{eq11A}), which gives us
the desired inequality~(\ref{eq10A}).
\end{remark}

\medskip
{\it Proof of Theorem~\ref{th1}}. Fix $y_0\in
\overline{f(D)}\setminus\{\infty\},$
$0<r_1<r_2<r_0=\sup\limits_{y\in f(D)}|y-y_0|,$ $C_1\subset B(y_0,
r_1)\cap f(D)$ and $C_2\subset f(D)\setminus B(y_0, r_2).$ Set
$$C_0:=\overline{f^{\,-1}(C_1)}\,,\quad
C^*_0:=\overline{f^{\,-1}(C_2)}$$ (see Figure~\ref{fig1}).
\begin{figure}[h]
\centering\includegraphics[scale=0.55]{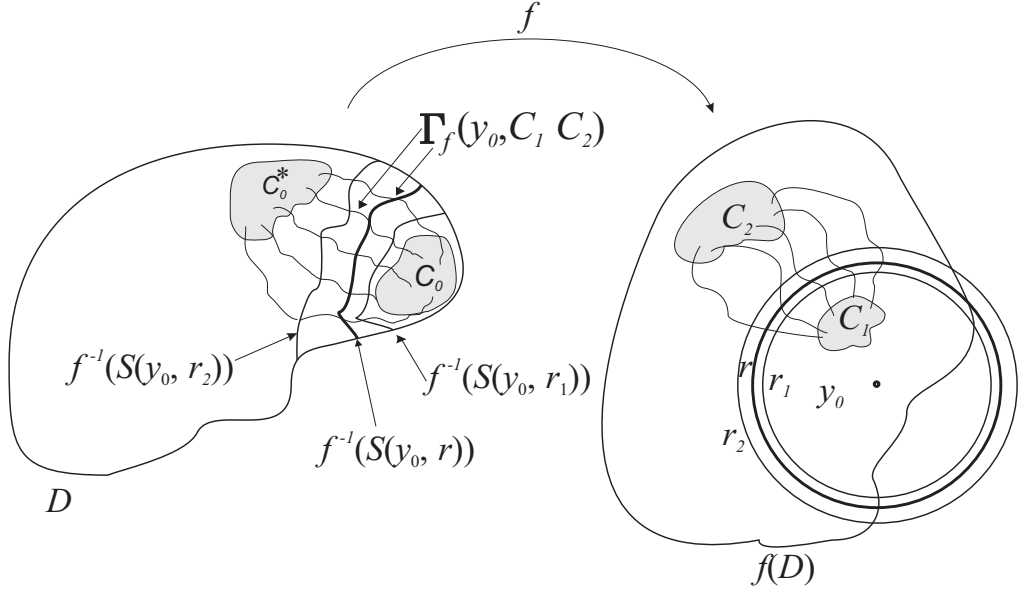}
\caption{To the proof of Theorem~\ref{th1}} \label{fig1}
\end{figure}
Observe that $C_0$ and $C_1$ are disjoint compact sets in $D,$ see
\cite[Theorem~3.3]{Vu}. Besides that, $C_1$ and $C_2$ are non empty
by the choose of $r_0,$ $r_1$ and $r_2.$

\medskip Let us to show that a set $\sigma_r:=f^{\,-1}(S(y_0, r))$
separates $C_0$ from $C^{\,*}_0$ in $D$ for any $r\in (r_1, r_2).$
Indeed, $\sigma_r$ is closed in $D$ as a preimage of a closed set
$S(y_0, r)$ under the continuous mapping~$f$ (see, e.g.,
\cite[Theorem~1.IV.13, Ch.~1]{Ku}). In particular, $\sigma_r$ is
also closed with respect to $R:=D\setminus (C_0\cup C^{\,*}_0).$ We
put
$$A:=f^{\,-1}(B(y_0,
r))$$ and
$$B:=D\setminus \overline{f^{\,-1}(B(y_0, r))}\,.$$
Observe that, $A$ and $B$ are not empty by the choice of $r_0,$
$r_1,$ $r_2$ and $r.$ Since $f$ is continuous, $f^{\,-1}(B(y_0, r))$
and $D\setminus \overline{f^{\,-1}(B(y_0, r))}$ are open in $D.$ In
other words, $A$ and $B$ are open in
$$R^{\,*}:=R\cup C_0\cup C_1=D\,.$$
Note that $A\cap B=\varnothing,$ and $R^{\,*}\setminus
\sigma_r=A\cup B.$ Let $\Sigma_{C_0, C^{\,*}_0}$ be the family of
all sets separating $C_0$ and $C^{\,*}_0$ in $R^{\,*}.$ In this
case, by the equations of Ziemer and Hesse, see (\ref{eq3})
and~(\ref{eq4}), respectively, we obtain that
\begin{equation}\label{eq7}
M_{\alpha}(\Gamma_f(y_0, C_1,
C_2))=(\widetilde{M}_{p/(n-1)(\Sigma_{r_1, r_2})})^{1-\alpha}\,,
\end{equation}
where $\alpha=\frac{p}{p-n+1}.$
Then by Lemma~\ref{lem1} and by the relation~(\ref{eq7}), we obtain
that
\begin{equation}\label{eq8}
M_{\alpha}(\Gamma_f(y_0, r_1, r_2))\leqslant
\left(\inf\limits_{\rho\in{\rm
ext\,adm}\,f(\Sigma_{\varepsilon})}\int\limits_{f(D)\cap A(y_0, r_1,
r_2)}\frac{\rho^p(y)}{N^{\frac{p}{n-1}}(f, D)\cdot
Q^{\frac{p-n+1}{n-1}}(y)}\,dm(y)\right)^{-\frac{n-1}{p-n+1}}\,,
\end{equation}
where $Q$ is defined by~(\ref{eq1}). Using the second remote formula
in the proof of Theorem~9.2 in \cite{MRSY}, we obtain that
$$\inf\limits_{\rho\in{\rm
ext\,adm}\,f(\Sigma_{\varepsilon})}\int\limits_{f(D)\cap A(y_0, r_1,
r_2)}\frac{\rho^p(y)}{N^{\frac{p}{n-1}}(f, D)\cdot
Q^{\frac{p-n+1}{n-1}}(y)}\,dm(y)=$$
\begin{equation}\label{eq9}
=\int\limits_{r_1}^{r_2} \left(\inf\limits_{\alpha\in
I(r)}\int\limits_{S(y_0, r)\cap
f(D)}\frac{\alpha^{q}(y)}{N^{\frac{p}{n-1}}(f, D)\cdot
Q^{\frac{p-n+1}{n-1}}(y)}\,\,\mathcal{H}^{n-1}(y)\right)\,dr\,,
\end{equation}
where $q=\frac{p}{n-1},$ and $I(r)$ denotes the set of all
measurable functions on $S(y_0, r)\cap f(D)$ such that
$\int\limits_{S(y_0, r)\cap f(D)}\alpha(x)\,\mathcal{H}^{n-1}=1.$
Then, choosing $X=S(y_0, r)\cap f(D),$ $\mu={\mathcal H}^{n-1}$ and
$\varphi=\frac{1}{Q}|_{S(y_0, r)\cap f(D)}$
in~\cite[Lemma~9.2]{MRSY}, we obtain that
\begin{equation}\label{eq10}
\int\limits_{r_1}^{r_2} \left(\inf\limits_{\alpha\in
I(r)}\int\limits_{S(y_0, r)\cap
f(D)}\frac{\alpha^q(y)}{Q(y)}\,\,d\mathcal{H}^{n-1}\right)\,dr=
\int\limits_{r_1}^{r_2}\frac{dr}{\Vert Q\Vert_s(r)}\,,
\end{equation}
where $\Vert Q\Vert_s(r)=\left(\int\limits_{S(y_0, r)\cap
f(D)}Q^s(x)\,d{\mathcal H}^{n-1}\right)^{1/s}$ and
$s:=\frac{n-1}{p-n+1}.$ Thus, by~(\ref{eq8}), (\ref{eq9}) and
(\ref{eq10}) we obtain that
$$M_{\alpha}(\Gamma_f(y_0, r_1, r_2))\leqslant N^{\alpha}(f, D)\cdot
\left(\int\limits_{r_1}^{r_2}\frac{dr}{\Vert
Q\Vert_1(r)}\right)^{\frac{n-1}{p-n+1}}=$$
\begin{equation}\label{eq11}
=\frac{{N^{\alpha}}(f, D)\cdot
\omega_{n-1}}{\left(\int\limits_{r_1}^{r_2}\frac{dr}{r^{\frac{n-1}{\alpha-1}}
\widetilde{q}^{1/(\alpha-1)}_{y_0}(r)}\right)
^{\frac{n-1}{{p-n+1}}}}=\frac{N^{\alpha}(f, D)\cdot
\omega_{n-1}}{\left(\int\limits_{r_1}^{r_2}\frac{dr}{r^{\frac{n-1}{\alpha-1}}
\widetilde{q}^{1/(\alpha-1)}_{y_0}(r)}\right)^{\alpha-1}}\,,
\end{equation}
where $q_{y_0}(r)=\frac{1}{\omega_{n-1}r^{n-1}}\int\limits_{S(y_0,
r)}\widetilde{Q}\,d\mathcal{H}^{n-1}$ and
$\widetilde{Q}(y)=\begin{cases}
Q(y)\,,&y\in f(D)\,,\\
0\,,&y\not\in f(D)\end{cases}.$
Finally, it follows from~(\ref{eq11}) and Proposition~\ref{pr1} that
the relation
$$M_{\frac{p}{p-n+1}}(\Gamma_f(y_0, r_1, r_2))\leqslant \int\limits_{A(y_0,r_1,r_2)\cap
f(D)}N^{\alpha}(f, D)\cdot Q(y)\cdot \eta^{\,\alpha} (|y-y_0|)\,
dm(y)$$
holds for a function $Q(y)=K_{O, \alpha}(y,
f^{\,-1}):=\sum\limits_{x\in f^{\,-1}(y)}K_{I, \alpha}(x, f),$ that
is desired conclusion.~$\Box$

\medskip
{\it Proof of Corollary~\ref{cor2}} immediately follows by
Theorem~\ref{th1} and additional arguments used under the proof of
Corollary~\ref{cor1}.~$\Box$

\medskip
\begin{remark}\label{rem3}
Observe that, the local and boundary behavior of mappings that
satisfy condition~(\ref{eq2*A}) is described in sufficient detail
in~\cite{SSD}, which makes it possible to transfer these results to
mappings participating in Theorem~\ref{th1}. Note also that the
mappings with the inverse Poletsky inequality are part of the
definition of quasiconformality in the case of a bounded function
$Q$ (see~\cite [Ch.~13.1]{Va}), and in the unbounded case were
obtained by different authors under different conditions for $Q$
(see, eg, \cite[Theorem~8.5]{MRSY}, \cite[Lemma~3.1]{Cr}, \cite{KR}
and \cite[Theorem~1.3]{Sev$_2$}). In particular, the statement below
follows directly from Theorem~\ref{th1} and
\cite[Theorem~4.1]{SevSkv}.

\medskip
For domains $D, D^{\,\prime}\subset {\Bbb R}^n,$ $n\geqslant 2,$ a
number $N\in {\Bbb N}$ and a Lebesgue measurable function $Q:{\Bbb
R}^n\rightarrow [0, \infty],$ $Q(y)\equiv 0$ for $y\in{\Bbb
R}^n\setminus D^{\,\prime},$ we denote by $\frak{R}_{Q, N}(D,
D^{\,\prime})$ the family of all open discrete mappings
$f:D\rightarrow D^{\,\prime}$ which are differentiable almost
everywhere, have $N$-Luzin property with respect to the Lebesgue
measure in ${\Bbb R}^n,$ satisfy relation~(\ref{eq2}) and have
$N^{\,- 1}$-property on $S(y_0, r)\cap D^{\,\prime}$ for almost all
$r\in(\varepsilon, r_0)$ relative to the Hausdorff measure
${\mathcal H}^{n-1}$ on $ S(y_0, r)$ for any $y_0\in D^{\,\prime}$
and $r_0=\sup\limits_{y\in D^{\,\prime}}|y-y_0|$ such that

\medskip
1) $N(f, D)\leqslant N,$

\medskip
2) $K_{I, n}(y, f^{\,-1})=\sum\limits_{x\in f^{\,-1}(y)}{K_{O, n}(x,
f)}\leqslant Q(y)$ for any $y\in D^{\,\prime}.$

\medskip
{\it If $Q\in L^1(D^{\,\prime}),$ $D^{\,\prime}$ is bounded and $K$
is a compact set in $D,$ then the inequality
\begin{equation}\label{eq2E}
|f(x)-f(y)|\leqslant\frac{C}{\log^{1/n}\left(1+\frac{r_*}{2|x-y|}\right)}
\end{equation}
holds for any $x, y\in K$ and all $f\in \frak{R, N}_Q(D,
D^{\,\prime}),$ where $C=C(n, N, K, \Vert Q\Vert_1, D,
D^{\,\prime})>0$ is some constant depending only on $n,$ $N,$ $K$
and $\Vert Q\Vert_1,$ $\Vert Q\Vert_1$ denotes $L^1$-norm of $Q$ in
$D^{\,\prime},$ and $r_*=d(K,
\partial D).$ }\end{remark}

\medskip
{\bf \noindent Oleksandr Dovhopiatyi} \\
{\bf 1.} Zhytomyr Ivan Franko State University,  \\
40 Bol'shaya Berdichevskaya Str., 10 008  Zhytomyr, UKRAINE \\
alexdov1111111@gmail.com

\medskip
\medskip
{\bf \noindent Evgeny Sevost'yanov} \\
{\bf 1.} Zhytomyr Ivan Franko State University,  \\
40 Bol'shaya Berdichevskaya Str., 10 008  Zhytomyr, UKRAINE \\
{\bf 2.} Institute of Applied Mathematics and Mechanics\\
of NAS of Ukraine, \\
1 Dobrovol'skogo Str., 84 100 Slavyansk,  UKRAINE\\
esevostyanov2009@gmail.com

\end{document}